\documentclass[11pt]{article}

\usepackage[T1]{fontenc}
\usepackage[utf8]{inputenc}
\usepackage{amsmath,amssymb,amsthm,mathtools}
\usepackage[margin=1in]{geometry}
\usepackage{microtype}
\usepackage[hidelinks]{hyperref}

\hypersetup{
  pdftitle={Opposed exponent profiles and real interpolation of measurable subspaces},
  pdfauthor={Asad Ullah}
}

\newtheorem{theorem}{Theorem}[section]
\newtheorem{proposition}[theorem]{Proposition}
\newtheorem{lemma}[theorem]{Lemma}
\newtheorem{corollary}[theorem]{Corollary}
\theoremstyle{remark}
\newtheorem{remark}[theorem]{Remark}

\newcommand{\1}{\mathbf 1}
\newcommand{\G}{\mathcal G}
\newcommand{\F}{\mathcal F}
\newcommand{\N}{\mathbb N}

\newcommand{\norm}[1]{\left\lVert#1\right\rVert}
\newcommand{\X}{\boldsymbol X}
\newcommand{\Y}{\boldsymbol Y}

\title{Opposed exponent profiles and real interpolation
of measurable subspaces}

\author{Asad Ullah\\[4pt]
\small Department of Mathematics and Statistics\\
\small University of North Carolina at Charlotte\\
\small Charlotte, North Carolina, USA\\
\small Corresponding author:
\href{mailto:aullah@uncc.edu}{\texttt{aullah@uncc.edu}}}

\date{}

\begin{document}
\maketitle

\begin{abstract}
We study real interpolation of the $\G$-measurable subspaces
$Y_j$ of variable Lebesgue spaces $L^{p_j(\cdot)}$. For a family
of continuous exponent pairs that vary in opposite directions on
the atoms of $\G$, we prove that $K$-closedness is equivalent to
a uniform logarithmic bound on the oscillation of the exponents.
The same bound characterizes the existence of a common bounded
projection and interpolation equality at a single interior
parameter. If it fails, then for each $0<\theta<1$ there is one
function belonging to all ambient interpolation spaces
$\X_{\theta,q}$, $0<q\leq\infty$, and to none of the corresponding
spaces $\Y_{\theta,q}$. Each endpoint subspace still admits a
bounded positive projection. The exponents can be chosen
arbitrarily close to prescribed constants $1<a<b<\infty$, with
constant sum and smoothness away from zero. We also give a
finite-profile criterion in terms of the endpoint maxima, compute
the associated $K$-functional loss, and obtain a counterexample
for adapted sequence spaces.
\end{abstract}


\section{Introduction}\label{sec:intro}

Let $X_0$ and $X_1$ be Banach function spaces on a probability
space, and let $\G$ be a completed sub-$\sigma$-algebra. The
subspaces $Y_j=X_j\cap L^0(\G)$ satisfy
\begin{equation}\label{eq:inclusion}
 (Y_0,Y_1)_{\theta,q}
 \longrightarrow (X_0,X_1)_{\theta,q}\cap L^0(\G),
 \qquad 0<\theta<1,\quad 0<q\leq\infty.
\end{equation}
If a projection onto the measurable subspace is bounded on both
endpoints, it can be applied to each term of an ambient
decomposition, giving equality in \eqref{eq:inclusion}. Conditional
expectation is the usual choice. We examine what happens when
the endpoint subspaces admit bounded projections but no common
projection is available.

For the general theory of interpolation of subspaces we refer to
\cite{BL}; the relation with invertibility for separately
complemented subspaces is studied in \cite{ACK}. Here the spaces
are $L^{p_j(\cdot)}$, and the subspaces consist of functions
constant on the atoms of $\G$. The obstruction comes from the
locations of the large values of the two exponents inside those
atoms.

We consider exponents of the form $p_0=a+d_nh$ and $p_1=b-d_nh$
on each atom $I_n$, after rescaling a fixed smooth function $h$
to that atom. Here $1<a<b<\infty$ and $d_n\to0$. The two
exponents attain their maxima on disjoint subsets, each occupying
a fixed fraction of $I_n$. Averaging over the appropriate subset
gives a bounded positive projection at each endpoint. An ambient
decomposition can also split the atom between the two endpoints,
whereas a $\G$-measurable decomposition must be constant on it.

For the profiles considered here, the criterion is
\[
 B=\sup_n d_n\log(1/\mu_n),\qquad \mu_n=\mathbb P(I_n).
\]
By Theorem~\ref{thm:main}, $B<\infty$ is equivalent to
$K$-closedness, the existence of a common bounded projection, and
equality in \eqref{eq:inclusion} at one interior parameter.
If $B=\infty$, then for each fixed $\theta$ we construct
\[
 f_\theta\in\bigcap_{0<q\leq\infty}(X_0,X_1)_{\theta,q}
 \quad\text{with}\quad
 f_\theta\notin\bigcup_{0<q\leq\infty}(Y_0,Y_1)_{\theta,q}.
\]
The exponents are continuous, have constant sum, and may be as
close as desired to $a$ and $b$ in the uniform norm.

The condition on $B$ is the form taken here by the atomic
oscillation condition of Jiao, Weisz, Wu and Zhou \cite{JWWZ}:
\begin{equation}\label{eq:knowncondition}
 \sup_A\mathbb P(A)^{p_-(A)-p_+(A)}<\infty.
\end{equation}
On $I_n$, the expression in \eqref{eq:knowncondition} is
$\mu_n^{-2d_n}$, so the condition is exactly $B<\infty$.
Lemma~\ref{lem:average} gives a direct proof that it bounds
conditional expectation. Necessity for our interpolation problem
comes from splitting each atom into two halves and comparing
that decomposition with the constant decompositions allowed in
$Y_0+Y_1$.

The case of an $L^\infty$ endpoint behaves differently.
Kempka and Vyb\'iral \cite[Theorem~8]{KV} identify
$(L^{p(\cdot)},L^\infty)_{\theta,q}$ with a variable Lorentz
space. Their counterexample to a variable-exponent Marcinkiewicz
theorem concerns boundedness of operators between interpolation
spaces. In the martingale setting, Jiao, Weisz, Wu and Zhou
\cite[Theorem~1.1]{JWWZinterp} prove an $L^\infty$-endpoint
identity for variable conditional Hardy spaces.
Ullah \cite[Theorem~1.1]{Ullah} proves the corresponding adapted
sequence-space identity for measurable exponents, without a
compatibility condition between the exponent and the filtration.
Randrianantoanina \cite[Theorems~3.1 and~3.6]{Rand} treats
martingale Hardy spaces over quasi-Banach K\"othe spaces; his
conditional Hardy-space result allows arbitrary such spaces at
the finite endpoint. The examples below concern two finite
variable endpoints.

Our projections are obtained by averaging over subsets where
the exponent is maximal. They are bounded and positive, with no
claim of contractivity. For positive contractive projections in
variable Lebesgue spaces, see Tasoev \cite{Tasoev}; complemented
subsymmetric basic sequences are studied by Ansorena and Bello
\cite{AB}.

After the preliminary estimates, we prove the continuous result
in Section~\ref{sec:continuous}. The finite-profile model in
Section~\ref{sec:finite} gives an explicit loss exponent.
The adapted-sequence example follows in Section~\ref{sec:adapted}
by placing the scalar example in the first coordinate.

\section{Definitions and elementary estimates}\label{sec:prelim}

All functions are identified modulo sets of measure zero.
Sub-$\sigma$-algebras are completed, and $L^0(\G)$ denotes the
space of equivalence classes having a $\G$-measurable representative.
For a measurable exponent $p$ and a set $A$ of positive measure,
write
\[
 p_-(A)=\operatorname*{ess\,inf}_{x\in A}p(x),
 \qquad
 p_+(A)=\operatorname*{ess\,sup}_{x\in A}p(x).
\]
Without $A$, the extrema are taken over the whole probability
space. For a measurable exponent satisfying
$1<p_-\leq p(x)\leq p_+<\infty$ on a probability space, set
\[
 \rho_p(f)=\int |f(x)|^{p(x)}\,d\mathbb P(x),
 \qquad
 \norm{f}_{p(\cdot)}
 =
 \inf\{\lambda>0:\rho_p(f/\lambda)\leq1\}.
\]
The resulting Banach function space is denoted by $L^{p(\cdot)}$.
We use its lattice property, completeness, and the equivalence
\[
 \norm f_{p(\cdot)}\leq1
 \Longleftrightarrow
 \rho_p(f)\leq1;
\]
see \cite{DHHR}. Multiplication by an indicator is contractive.
Norm convergence implies convergence in measure, so
\[
 L^{p(\cdot)}\cap L^0(\G)
\]
is closed: an almost everywhere convergent subsequence of a
norm-convergent sequence has a $\G$-measurable limit.

For positive quantities $U,V$, the notation $U\asymp V$ means
that $cV\leq U\leq CV$ with positive constants independent of the
index under consideration. We write $\X_{\theta,q}$ for
$(X_0,X_1)_{\theta,q}$ and use the analogous convention for other
couples. For a Banach couple $\X=(X_0,X_1)$ we use
\[
 K(t,f;\X)=
 \inf_{f=f_0+f_1}
 \bigl(\norm{f_0}_{X_0}+t\norm{f_1}_{X_1}\bigr),
 \qquad t>0,
\]
and, for $0<\theta<1$ and $0<q\leq\infty$,
\[
 \norm f_{\X_{\theta,q}}
 =
 \begin{cases}
 \displaystyle
 \left(\int_0^\infty
 [t^{-\theta}K(t,f;\X)]^q\,\frac{dt}{t}\right)^{1/q},
 &q<\infty,\\[6pt]
 \displaystyle
 \sup_{t>0}t^{-\theta}K(t,f;\X),
 &q=\infty.
 \end{cases}
\]
For closed subspaces $Y_j\subset X_j$ with the inherited norms,
the subcouple $\Y=(Y_0,Y_1)$ is $K$-closed in $\X$ if a constant
$C$ satisfies
\[
 K(t,f;\Y)\leq C K(t,f;\X)
 \qquad(f\in Y_0+Y_1,\ t>0).
\]
The opposite inequality has constant one. A common bounded
projection means a linear map
\[
 P:X_0+X_1\longrightarrow Y_0+Y_1
\]
that fixes $Y_0+Y_1$ and maps each $X_j$ boundedly into $Y_j$.

Put
\begin{equation}\label{eq:Ctheta}
 c_{\theta,q}=
 \begin{cases}
 [1/(q\theta(1-\theta))]^{1/q},&q<\infty,\\
 1,&q=\infty.
 \end{cases}
\end{equation}
Direct integration gives, for $u,v>0$,
\begin{equation}\label{eq:minintegral}
 \norm{t^{-\theta}\min(u,tv)}_{L^q(dt/t)}
 =
 c_{\theta,q}u^{1-\theta}v^\theta.
\end{equation}
For $q=\infty$ the left side denotes the supremum.

\begin{lemma}[Indicators and atoms]\label{lem:atoms}
Let $A$ have measure $0<\mu\leq1$. If $p(x)\leq P$ on $A$, then
\[
 \norm{\1_A}_{p(\cdot)}\leq\mu^{1/P}.
\]
If also $p=P$ on a subset of $A$ of measure at least $\kappa\mu$,
then
\begin{equation}\label{eq:indicator}
 (\kappa\mu)^{1/P}
 \leq\norm{\1_A}_{p(\cdot)}
 \leq\mu^{1/P}.
\end{equation}
Suppose that $A$ is an atom of $\G$, $X_j=L^{p_j(\cdot)}$, and
$Y_j=X_j\cap L^0(\G)$. With $a_j=\norm{\1_A}_{X_j}$, one has
\begin{equation}\label{eq:atomK}
 \begin{split}
 K(t,z\1_A;\Y)&=|z|\min(a_0,ta_1),\\
 \norm{z\1_A}_{\Y_{\theta,q}}
 &=c_{\theta,q}|z|a_0^{1-\theta}a_1^\theta.
 \end{split}
\end{equation}
Moreover, if $f\in Y_0+Y_1$ equals $z$ on $A$, its intrinsic
interpolation norm is at least the last expression.
\end{lemma}

\begin{proof}
For $\lambda=\mu^{1/P}\leq1$,
\[
 \int_A\lambda^{-p(x)}\,d\mathbb P
 \leq\mu\lambda^{-P}=1.
\]
On the subset where $p=P$, any admissible $\lambda$ satisfies
$\kappa\mu\lambda^{-P}\leq1$, proving \eqref{eq:indicator}.

A decomposition in $\Y$ restricts on $A$ to constants $z_0,z_1$
with $z_0+z_1=z$. Its cost is at least
\[
 a_0|z_0|+ta_1|z_1|
 \geq\min(a_0,ta_1)|z|.
\]
Putting the whole function into either endpoint proves equality.
Equation~\eqref{eq:minintegral} gives the norm identity. Finally,
multiplication by $\1_A$ is contractive on both $Y_j$ and hence
decreases their $K$-functional.
\end{proof}

\begin{lemma}[Summation]\label{lem:sum}
Let $r=\min(1,q)$, with $r=1$ for $q=\infty$. For a finite family,
\begin{equation}\label{eq:rtriangle}
 \norm{\sum_k f_k}_{\X_{\theta,q}}^r
 \leq\sum_k\norm{f_k}_{\X_{\theta,q}}^r.
\end{equation}
The same estimate holds for a series whose partial sums converge
in $X_0+X_1$, whenever the right side is finite. In particular,
\begin{equation}\label{eq:singleupper}
 \norm f_{\X_{\theta,q}}
 \leq
 c_{\theta,q}
 \norm f_{X_0}^{1-\theta}\norm f_{X_1}^{\theta}
 \qquad(f\in X_0\cap X_1).
\end{equation}
\end{lemma}

\begin{proof}
The $K$-functional is subadditive. For $q\geq1$, use Minkowski's
inequality, including the supremum version for $q=\infty$. For
$q<1$, use
\[
 \left(\sum b_k\right)^q\leq\sum b_k^q
\]
for nonnegative $b_k$.

If $S_N\to f$ in the sum space, then, for fixed $t$,
\[
 \begin{split}
 |K(t,S_N;\X)-K(t,f;\X)|
 &\leq K(t,S_N-f;\X)\\
 &\leq\max(1,t)K(1,S_N-f;\X)
 \longrightarrow0.
 \end{split}
\]
Fatou's lemma, or taking suprema when $q=\infty$, proves the series
assertion. The estimate in \eqref{eq:singleupper} follows by placing
$f$ entirely in either endpoint and applying \eqref{eq:minintegral}.
\end{proof}

We shall use two averaging operators: conditional expectation and
a projection obtained by averaging where the exponent is maximal.

\begin{lemma}[Averaging under atomic oscillation control]
\label{lem:average}
Let $\G$ be generated, modulo null sets, by a countable partition
$(A_n)$ into sets of positive measure $\mu_n$. Suppose
$1<r_0\leq p\leq s_0<\infty$ and
\[
 B_p=
 \sup_n[p_+(A_n)-p_-(A_n)]\log(1/\mu_n)<\infty.
\]
Then the usual conditional expectation $E_\G$ is bounded on
$L^{p(\cdot)}$.
\end{lemma}

\begin{proof}
Normalize $\rho_p(f)\leq1$. Write
\[
 r_n=p_-(A_n),\quad s_n=p_+(A_n),\quad
 v_n=\int_{A_n}|f|^{p(x)}\,d\mathbb P,\quad
 m_n=\mu_n^{-1}\int_{A_n}f\,d\mathbb P.
\]
Since $|u|^{r_n}\leq1+|u|^{p(x)}$ on $A_n$, Jensen's inequality gives
\[
 |m_n|^{r_n}\leq1+v_n/\mu_n\leq2/\mu_n.
\]
If $|m_n|>1$, then
\begin{align*}
 \int_{A_n}|m_n|^{p(x)}\,d\mathbb P
 &\leq\mu_n|m_n|^{s_n}\\
 &\leq(\mu_n+v_n)(2/\mu_n)^{(s_n-r_n)/r_n}\\
 &\leq D(\mu_n+v_n),
 \qquad D=2^{s_0/r_0}e^{B_p/r_0}.
\end{align*}
If $|m_n|\leq1$, the same bound follows from the upper bound $\mu_n$.
Summing gives $\rho_p(E_\G f)\leq2D$. Consequently
\[
 \norm{E_\G f}_{p(\cdot)}\leq(2D)^{1/r_0}.
\]
Homogeneity proves boundedness. All averages exist because
$L^{p(\cdot)}\subset L^1$ on a probability space when $p_->1$.
\end{proof}

\begin{lemma}[Projection from a maximal-exponent region]
\label{lem:plateau}
Let $\G$ be generated, modulo null sets, by a countable partition
$(A_n)$ into sets of positive measure, and suppose that
$1<r_0\leq p\leq s_0<\infty$. For each $n$, let $E_n\subset A_n$
satisfy
\[
 \mathbb P(E_n)\geq\kappa\mathbb P(A_n),\qquad
 p(x)=P_n\text{ on }E_n,\qquad
 p(x)\leq P_n\text{ on }A_n,
\]
where $\kappa>0$ is independent of $n$. Then
\[
 Rf=
 \sum_n
 \left(\frac1{\mathbb P(E_n)}\int_{E_n}f\,d\mathbb P\right)
 \1_{A_n}
\]
is a positive bounded projection from $L^{p(\cdot)}$ onto its
$\G$-measurable subspace, with norm at most
$(1+\kappa^{-1})^{1/r_0}$.
\end{lemma}

\begin{proof}
Let $m_n$ be the displayed average and normalize $\rho_p(f)\leq1$.
For $|m_n|>1$, Jensen's inequality on $E_n$ gives
\[
 \begin{split}
 \int_{A_n}|m_n|^{p(x)}\,d\mathbb P
 &\leq\mathbb P(A_n)|m_n|^{P_n}\\
 &\leq\kappa^{-1}\int_{E_n}|f|^{P_n}\,d\mathbb P.
 \end{split}
\]
The sum of the contributions with $|m_n|\leq1$ is at most one.
Thus $\rho_p(Rf)\leq1+\kappa^{-1}$, which gives the norm bound.
The formula defines a linear positive operator and fixes every
$\G$-measurable function in the space, proving the assertion.
\end{proof}

\section{A criterion for opposed continuous profiles}
\label{sec:continuous}

Fix constants
\begin{equation}\label{eq:parameters}
 1<a<b<\infty,\qquad
 0<\varepsilon<\min\{a-1,(b-a)/2\}.
\end{equation}
Let $(\mu_n)_{n\geq1}$ and $(d_n)_{n\geq1}$ satisfy
\begin{equation}\label{eq:sequences}
 \mu_n>0,\quad \sum_{n\geq1}\mu_n<1,\qquad
 0<d_n\leq\varepsilon,\quad d_n\longrightarrow0.
\end{equation}
Set
\[
 s_n=\sum_{k\geq n}\mu_k,\qquad
 I_n=(s_{n+1},s_n),\qquad I_0=(s_1,1).
\]
Let $\G$ be generated by these intervals and all Lebesgue null sets
on $[0,1]$.

Choose a smooth function $h$ on $[0,1]$, vanishing near both
endpoints, such that
\begin{equation}\label{eq:h}
 \begin{gathered}
 -1\leq h\leq1,\qquad
 h\geq0\text{ on }[0,1/2],\quad
 h\leq0\text{ on }[1/2,1],\\
 h=1\text{ on }[1/8,3/8],\qquad
 h=-1\text{ on }[5/8,7/8].
 \end{gathered}
\end{equation}
For completeness, put $\eta(t)=e^{-1/t}$ for $t>0$ and $\eta(t)=0$
otherwise,
\[
 S(t)=\frac{\eta(t)}{\eta(t)+\eta(1-t)},
\]
and
\[
 \psi(u)=S(16u-1)S(7-16u),\qquad
 h(u)=\psi(u)-\psi(1-u).
\]
Then $h$ satisfies \eqref{eq:h}.

Define
\begin{equation}\label{eq:exponents}
 \begin{split}
 p_0(x)&=a+d_nh((x-s_{n+1})/\mu_n),\\
 p_1(x)&=b-d_nh((x-s_{n+1})/\mu_n),
 \qquad x\in I_n,\ n\geq1,
 \end{split}
\end{equation}
and set $p_0=a$, $p_1=b$ at all other points. Finally, write
\[
 X_j=L^{p_j(\cdot)}([0,1]),\qquad
 Y_j=X_j\cap L^0(\G),\qquad
 B=\sup_{n\geq1}d_n\log(1/\mu_n).
\]

\begin{theorem}\label{thm:main}
For the preceding construction, $p_0,p_1$ are continuous on $[0,1]$
and smooth on $(0,1]$. They satisfy
\[
 p_0+p_1=a+b,\qquad
 \norm{p_0-a}_\infty,\norm{p_1-b}_\infty\leq\varepsilon,
 \qquad 1<(p_0)_+<(p_1)_-<\infty.
\]
Each $Y_j$ is the range of a bounded positive projection on $X_j$,
with bounds depending only on $a,b,\varepsilon$.

The following conditions are equivalent:
\begin{enumerate}
 \item $B<\infty$.
 \item $E_\G$ is bounded on both $X_0$ and $X_1$.
 \item There is a common bounded projection for the subcouple
 $\Y\subset\X$.
 \item $\Y$ is $K$-closed in $\X$.
 \item For at least one $0<\theta<1$ and $0<q\leq\infty$,
 \[
 \Y_{\theta,q}=\X_{\theta,q}\cap L^0(\G)
 \]
 as sets.
 \item The preceding equality holds with equivalent quasi-norms for
 every $0<\theta<1$ and $0<q\leq\infty$.
\end{enumerate}
If $B=\infty$, then for each $0<\theta<1$ there exists
$f_\theta\in Y_0$ such that
\begin{equation}\label{eq:strongfailure}
 f_\theta\in\bigcap_{0<q\leq\infty}\X_{\theta,q},
 \qquad
 f_\theta\notin\bigcup_{0<q\leq\infty}\Y_{\theta,q}.
\end{equation}
Thus the inclusion in \eqref{eq:inclusion} is strict for every
parameter.
\end{theorem}

\begin{proof}
Continuity at the endpoints of each $I_n$ follows because $h$
vanishes in neighborhoods of $0$ and $1$. The only accumulation
point of the interval endpoints is zero, and $d_n\to0$ proves
continuity there. The same observation proves smoothness away from
zero. The remaining pointwise assertions follow from
\eqref{eq:parameters}.

On $I_n$, $p_0$ attains its maximum $a+d_n$ on a subinterval of
length $\mu_n/4$, while $p_1$ attains its maximum $b+d_n$ on a
different subinterval of the same length. On $I_0$ both exponents
are constant. Apply Lemma~\ref{lem:plateau} separately to the two
exponents, with $\kappa=1/4$. The resulting endpoint projections
are bounded independently of $B$.

The oscillation of either exponent on $I_n$ is $2d_n$.
Consequently Lemma~\ref{lem:average} proves
(1)$\Rightarrow$(2). Conditional expectation fixes
$\G$-measurable functions, so (2)$\Rightarrow$(3). Applying a common
projection to an arbitrary ambient decomposition proves
(3)$\Rightarrow$(4), with constant equal to the larger endpoint
operator norm.

We have $X_1\subset X_0$ continuously: if
$\rho_{p_1}(f)\leq1$, then
\[
 \rho_{p_0}(f)
 \leq
 \int_{\{|f|\leq1\}}1\,dx+
 \int_{\{|f|>1\}}|f|^{p_1(x)}\,dx
 \leq2.
\]
Hence $X_0+X_1=X_0$ and $Y_0+Y_1=Y_0$. Every $\G$-measurable
element of $\X_{\theta,q}$ therefore belongs to $Y_0+Y_1$.
Integration of the $K$-closedness inequality proves
(4)$\Rightarrow$(6). Clearly (6)$\Rightarrow$(5).

We prove \eqref{eq:strongfailure} when $B=\infty$; this will also
prove (5)$\Rightarrow$(1). Write
\[
 u_n=\norm{\1_{I_n}}_{X_0},\qquad
 v_n=\norm{\1_{I_n}}_{X_1},
\]
\[
 \alpha_n=
 \frac{1-\theta}{a+d_n}+\frac{\theta}{b+d_n},
 \qquad
 M_n=u_n^{1-\theta}v_n^\theta.
\]
Lemma~\ref{lem:atoms}, and the fact that all exponents exceed one,
give
\begin{equation}\label{eq:uv}
 \begin{split}
 \tfrac14\mu_n^{1/(a+d_n)}
 &\leq u_n\leq\mu_n^{1/(a+d_n)},\\
 \tfrac14\mu_n^{1/(b+d_n)}
 &\leq v_n\leq\mu_n^{1/(b+d_n)}.
 \end{split}
\end{equation}
In particular,
\begin{equation}\label{eq:M}
 \tfrac14\mu_n^{\alpha_n}\leq M_n\leq\mu_n^{\alpha_n}.
\end{equation}

Split $I_n$ into its left and right halves $I_n^+$ and $I_n^-$.
On the left half, $p_0\leq a+d_n$ and $p_1\leq b$; on the right
half, $p_0\leq a$ and $p_1\leq b+d_n$. Applying the upper indicator
estimate and \eqref{eq:singleupper} separately to the two halves
gives
\begin{align*}
 \norm{\1_{I_n^+}}_{\X_{\theta,q}}
 &\leq
 c_{\theta,q}
 \mu_n^{(1-\theta)/(a+d_n)+\theta/b},\\
 \norm{\1_{I_n^-}}_{\X_{\theta,q}}
 &\leq
 c_{\theta,q}
 \mu_n^{(1-\theta)/a+\theta/(b+d_n)}.
\end{align*}
With $r=\min(1,q)$, interpreted as one when $q=\infty$, set
\[
 \gamma_n=
 \min\left\{
 \theta\left(\frac1b-\frac1{b+d_n}\right),
 (1-\theta)\left(\frac1a-\frac1{a+d_n}\right)
 \right\}.
\]
Lemma~\ref{lem:sum} and \eqref{eq:M} imply
\begin{equation}\label{eq:ambientblock}
 \norm{M_n^{-1}\1_{I_n}}_{\X_{\theta,q}}
 \leq
 4\,2^{1/r}c_{\theta,q}\mu_n^{\gamma_n}.
\end{equation}
Furthermore,
\begin{equation}\label{eq:gammalower}
 \gamma_n\geq k_\theta d_n,\qquad
 k_\theta=
 \min\left\{
 \frac{\theta}{b(b+\varepsilon)},
 \frac{1-\theta}{a(a+\varepsilon)}
 \right\}>0.
\end{equation}
Since $B=\infty$, choose strictly increasing indices $n_k$ such
that
\begin{equation}\label{eq:subsequence}
 d_{n_k}\log(1/\mu_{n_k})\geq k^2,
 \qquad
 \log(1/\mu_{n_k})\geq k^2.
\end{equation}
Such a choice is possible because an unbounded sequence remains
unbounded after deletion of finitely many terms, and
$d_n\leq\varepsilon$. Define
\begin{equation}\label{eq:f}
 f_\theta=
 \sum_{k=1}^{\infty}
 \frac{k}{M_{n_k}}\1_{I_{n_k}}.
\end{equation}
The sum is pointwise well defined on disjoint intervals. Let
\[
 d=\frac1{a+\varepsilon}-\frac1b>0.
\]
By \eqref{eq:uv},
\[
 \begin{split}
 \norm{\frac{k}{M_{n_k}}\1_{I_{n_k}}}_{X_0}
 &=k(u_{n_k}/v_{n_k})^\theta\\
 &\leq4^\theta k\mu_{n_k}^{d\theta}\\
 &\leq4^\theta k e^{-d\theta k^2}.
 \end{split}
\]
Thus the series converges in $X_0$ and its limit is
$f_\theta\in Y_0$. Equations
\eqref{eq:ambientblock}--\eqref{eq:subsequence} give
\[
 \norm{\frac{k}{M_{n_k}}\1_{I_{n_k}}}_{\X_{\theta,q}}
 \leq
 4\,2^{1/r}c_{\theta,q}k e^{-k_\theta k^2}.
\]
The $r$th powers are summable for every $q>0$; for $q=\infty$ the
displayed bounds themselves are summable.
Lemma~\ref{lem:sum} proves membership in every ambient
interpolation space, using the same $f_\theta$ for all $q$.

On the other hand, restriction to $I_{n_k}$ and
\eqref{eq:atomK} give
\[
 \norm{f_\theta}_{\Y_{\theta,q}}
 \geq
 c_{\theta,q}
 \frac{k}{M_{n_k}}
 u_{n_k}^{1-\theta}v_{n_k}^{\theta}
 =
 c_{\theta,q}k.
\]
Letting $k\to\infty$ proves nonmembership for every $q$, including
$q=\infty$. This proves \eqref{eq:strongfailure} and
(5)$\Rightarrow$(1).
\end{proof}

\begin{proposition}[Quantitative loss of $K$-closedness]
\label{prop:loss}
In the construction above, write
$u_n=\norm{\1_{I_n}}_{X_0}$ and
$v_n=\norm{\1_{I_n}}_{X_1}$, and let $t_n=u_n/v_n$ and
\[
 k_*=
 \min\left\{
 \frac1{a(a+\varepsilon)},
 \frac1{b(b+\varepsilon)}
 \right\}.
\]
Then
\begin{equation}\label{eq:Kratio}
 \frac{K(t_n,\1_{I_n};\Y)}
      {K(t_n,\1_{I_n};\X)}
 \geq
 \frac18
 \exp\bigl(k_*d_n\log(1/\mu_n)\bigr).
\end{equation}
\end{proposition}

\begin{proof}
The numerator equals $u_n$ by \eqref{eq:atomK}. Place
$\1_{I_n^-}$ in $X_0$ and $\1_{I_n^+}$ in $X_1$. Their endpoint
norms are at most $\mu_n^{1/a}$ and $\mu_n^{1/b}$, respectively.
Hence
\[
 \begin{split}
 \frac{K(t_n,\1_{I_n};\X)}{u_n}
 &\leq
 \frac{\mu_n^{1/a}}{u_n}
 +
 \frac{\mu_n^{1/b}}{v_n}\\
 &\leq
 4\mu_n^{1/a-1/(a+d_n)}
 +
 4\mu_n^{1/b-1/(b+d_n)}\\
 &\leq8\mu_n^{k_*d_n}.
 \end{split}
\]
Taking reciprocals proves \eqref{eq:Kratio}.
\end{proof}

\begin{corollary}[Arbitrarily small continuous perturbations]
\label{cor:continuous}
For every $1<a<b<\infty$ and every $\delta>0$, there are continuous
exponents within $\delta$ of $a$ and $b$, respectively, and a
completed sub-$\sigma$-algebra on $[0,1]$ for which both endpoint
subspaces are positively complemented and \eqref{eq:strongfailure}
holds for every $0<\theta<1$. The exponents can be smooth away from
zero and can satisfy $p_0+p_1=a+b$.
\end{corollary}

\begin{proof}
Choose $\varepsilon<\delta$ satisfying \eqref{eq:parameters} and set
\[
 \mu_n=e^{-n^2}/4,\qquad d_n=\varepsilon/(n+1).
\]
Then $\sum_n\mu_n<1$ and
\[
 d_n\log(1/\mu_n)
 =
 \frac{\varepsilon(n^2+\log4)}{n+1}
 \longrightarrow\infty.
\]
Apply Theorem~\ref{thm:main}.
\end{proof}

\begin{remark}
For the same lengths $\mu_n=e^{-n^2}/4$, choosing
$d_n=\varepsilon/(n+1)^2$ instead makes $B$ finite, and lifting holds.
Thus convergence of the exponent oscillations to zero does not
decide the question. Their rate relative to atom size does.

Also, on each $I_n$ condition \eqref{eq:knowncondition} is exactly
$\mu_n^{-2d_n}\leq C$. The necessity proved here applies to this
opposed profile family, not to arbitrary exponent pairs.
\end{remark}

\begin{remark}[Continuity and log-H\"older regularity]
The continuity in Theorem~\ref{thm:main} is weaker than the usual
local log-H\"older condition
\[
 |p(x)-p(y)|\leq
 \frac{C}{\log(e+1/|x-y|)}\qquad(x\ne y).
\]
Indeed, choose points $x_n,y_n$ in the two plateau regions of
$I_n$. For either exponent the difference is $2d_n$, while
$|x_n-y_n|\leq\mu_n$. A local log-H\"older bound would therefore give
\[
 2d_n\log(1/\mu_n)
 \leq2d_n\log(e+1/\mu_n)\leq C.
\]
Consequently the examples with $B=\infty$ are not locally
log-H\"older continuous. Their continuity follows instead from
$d_n\to0$.
\end{remark}

\section{A finite-profile classification}\label{sec:finite}

We now take the same finite exponent profile on every atom.
Let $\mu_n>0$, $\sum_n\mu_n=1$, and let
$\pi_1,\ldots,\pi_m>0$ satisfy $\sum_\ell\pi_\ell=1$. On
\[
 \Omega=\N\times\{1,\ldots,m\}
\]
give $(n,\ell)$ mass $\mu_n\pi_\ell$. Let $\G$ be the
$\sigma$-algebra generated by the fibres
\[
 A_n=\{n\}\times\{1,\ldots,m\}.
\]
Choose numbers $p_{j\ell}$ satisfying
\[
 1<\min_\ell p_{0\ell}
 \leq A_0:=\max_\ell p_{0\ell}
 <\min_\ell p_{1\ell}
 \leq A_1:=\max_\ell p_{1\ell}<\infty.
\]
Define
\[
 p_j(n,\ell)=p_{j\ell},\qquad
 X_j=L^{p_j(\cdot)},\qquad
 Y_j=X_j\cap L^0(\G).
\]

\begin{theorem}\label{thm:finite}
The following are equivalent:
\begin{enumerate}
 \item There exists $\ell_*$ such that
 $p_{0\ell_*}=A_0$ and $p_{1\ell_*}=A_1$.
 \item There is a common bounded projection for $\Y\subset\X$.
 \item $\Y$ is $K$-closed in $\X$.
 \item For at least one interior parameter $(\theta,q)$,
 \[
 \Y_{\theta,q}=\X_{\theta,q}\cap L^0(\G)
 \]
 as sets.
\end{enumerate}
When these conditions hold, equality holds with equivalent
quasi-norms at every interior parameter. When they fail, for each
$\theta$ there is a function satisfying \eqref{eq:strongfailure}.

There is also an explicit power law for the $K$-functional loss.
Put
\[
 d=1/A_0-1/A_1,\qquad t_n=\mu_n^d,
\]
and
\[
 \beta=
 \min_{1\leq\ell\leq m}
 \max\{1/p_{0\ell},d+1/p_{1\ell}\}.
\]
Then, with constants independent of $n$,
\begin{equation}\label{eq:finiteK}
 \frac{K(t_n,\1_{A_n};\Y)}
      {K(t_n,\1_{A_n};\X)}
 \asymp
 \mu_n^{-(\beta-1/A_0)}.
\end{equation}
The exponent on the right is positive precisely when condition
(1) fails.
\end{theorem}

\begin{proof}
If $g=(g_n)$ is regarded as a diagonal function
$Dg(n,\ell)=g_n$, then
\begin{equation}\label{eq:diagonal}
 c_j\norm g_{L^{A_j}(\mu)}
 \leq\norm{Dg}_{X_j}
 \leq\norm g_{L^{A_j}(\mu)},
 \qquad
 c_j=\pi_{\ell_j}^{1/A_j},
\end{equation}
where $\ell_j$ is any coordinate with $p_{j\ell_j}=A_j$.
For the lower estimate retain the $\ell_j$ term of the modular.
For the upper estimate normalize
$\norm g_{L^{A_j}(\mu)}=1$. Since $\mu$ is a probability measure
and $p_{j\ell}\leq A_j$, H\"older's inequality gives
\[
 \sum_n\mu_n|g_n|^{p_{j\ell}}\leq1.
\]
The weighted sum of these inequalities is the required modular
bound.

Under (1), the map
\[
 Pf(n,\ell)=f(n,\ell_*)
\]
is a projection onto diagonal functions. Restricting the modular to
coordinate $\ell_*$ gives
\[
 \norm{f(\cdot,\ell_*)}_{L^{A_j}(\mu)}
 \leq
 \pi_{\ell_*}^{-1/A_j}\norm f_{X_j}.
\]
Together with \eqref{eq:diagonal}, this proves (2). The implications
(2)$\Rightarrow$(3) and lifting at all parameters follow as in
Theorem~\ref{thm:main}; here again $X_1\subset X_0$.

Suppose (1) fails. Fix $\theta$ and put
\[
 \alpha=\frac{1-\theta}{A_0}+\frac\theta{A_1},
\]
\[
 \delta_\theta=
 \min_\ell
 \left[
 (1-\theta)
 \left(\frac1{p_{0\ell}}-\frac1{A_0}\right)
 +
 \theta
 \left(\frac1{p_{1\ell}}-\frac1{A_1}\right)
 \right].
\]
Each bracket is positive: its two nonnegative terms could vanish
together only at a common maximizing coordinate. Since there are
finitely many coordinates, $\delta_\theta>0$.

Set
\[
 u_n=\norm{\1_{A_n}}_{X_0},\qquad
 v_n=\norm{\1_{A_n}}_{X_1},\qquad
 M_n=u_n^{1-\theta}v_n^\theta.
\]
By \eqref{eq:diagonal},
\[
 u_n\asymp\mu_n^{1/A_0},\qquad
 v_n\asymp\mu_n^{1/A_1},\qquad
 M_n\asymp\mu_n^\alpha.
\]
The indicator of the single point $(n,\ell)$ has endpoint norm
$(\mu_n\pi_\ell)^{1/p_{j\ell}}$. Summing its interpolation bounds
from \eqref{eq:singleupper} over the finitely many coordinates yields
\begin{equation}\label{eq:finiteblock}
 \norm{M_n^{-1}\1_{A_n}}_{\X_{\theta,q}}
 \leq C_{\theta,q}\mu_n^{\delta_\theta}.
\end{equation}

Choose increasing indices $n_k$ with $\mu_{n_k}\leq e^{-k^2}$,
which is possible because $\mu_n\to0$, and put
\[
 f_\theta=
 \sum_k kM_{n_k}^{-1}\1_{A_{n_k}}.
\]
The $X_0$-norms of its summands are bounded by
\[
 Ck\mu_{n_k}^{\theta(1/A_0-1/A_1)},
\]
so the series converges in $Y_0$.
Equation~\eqref{eq:finiteblock} and Lemma~\ref{lem:sum} place it in
every ambient interpolation space. Lemma~\ref{lem:atoms} gives the
intrinsic lower bound $c_{\theta,q}k$ for each $k$, proving
\eqref{eq:strongfailure} and disproving (4).

To compute the loss in \eqref{eq:finiteK}, restrict to one point
for a lower bound and decompose into single points for an upper
bound. This gives
\begin{align*}
 &\max_\ell
 \min\{(\mu_n\pi_\ell)^{1/p_{0\ell}},
             t(\mu_n\pi_\ell)^{1/p_{1\ell}}\}\\
 &\hspace{25mm}\leq K(t,\1_{A_n};\X)\\
 &\hspace{25mm}\leq
 \sum_\ell
 \min\{(\mu_n\pi_\ell)^{1/p_{0\ell}},
             t(\mu_n\pi_\ell)^{1/p_{1\ell}}\}.
\end{align*}
For the upper bound assign each point to its cheaper endpoint and
use the triangle inequality. Since the profile is finite and fixed,
substituting $t=t_n$ shows that the ambient $K$-functional is
comparable to $\mu_n^\beta$.

Equations~\eqref{eq:atomK} and \eqref{eq:diagonal} show that the
intrinsic one is comparable to $\mu_n^{1/A_0}$. The quotient
therefore has the order stated in \eqref{eq:finiteK}. Both terms
in each maximum defining $\beta$ are at least $1/A_0$. The maximum
equals $1/A_0$ exactly at a common maximizing coordinate, which
proves the final assertion.
\end{proof}

\begin{remark}[The finite direct-sum model]
Up to equivalent norms, $X_j$ is the finite direct sum of
$L^{p_{j\ell}}(\mu)$, and \eqref{eq:diagonal} identifies $Y_j$
with $L^{A_j}(\mu)$. For fixed $\theta$, write
\[
 \frac1{r_\ell}=\frac{1-\theta}{p_{0\ell}}
                  +\frac\theta{p_{1\ell}},
 \qquad
 \frac1{r_*}=\frac{1-\theta}{A_0}+\frac\theta{A_1}.
\]
These are the exponents in the classical real interpolation
formulas for the coordinate couples and the diagonal couple
\cite{BL}. We have $r_\ell\leq r_*$, with equality precisely at
a common maximizing coordinate. The common-coordinate criterion
thus agrees with the classical interpolation formulas. The block
proof additionally yields \eqref{eq:finiteK} and the same
separating function for every $q$.
\end{remark}

\begin{remark}[A numerical profile]
Take $m=2$, $\pi_1=\pi_2=1/2$, and
\[
 (p_{01},p_{02})=(3,2),\qquad
 (p_{11},p_{12})=(4,8).
\]
Then
\[
 A_0=3,\qquad A_1=8,\qquad
 d=5/24,\qquad \beta=11/24.
\]
The ratio in \eqref{eq:finiteK} is comparable to $\mu_n^{-1/8}$.
The interpolation gap in the proof is
\[
 \delta_\theta
 =
 \min\{\theta/8,(1-\theta)/6\}>0.
\]
Each endpoint admits a bounded projection obtained by copying its
own maximal-exponent coordinate, but no common bounded projection
exists.
\end{remark}

\section{Adapted sequence spaces}\label{sec:adapted}

Let $\F$ be the completed Lebesgue $\sigma$-algebra on $[0,1]$ and
define a filtration by
\[
 \F_1=\G,\qquad \F_n=\F\quad(n\geq2).
\]
For the exponents in Theorem~\ref{thm:main}, put
\[
 \mathcal X_j=L^{p_j(\cdot)}([0,1];\ell_2),
\]
\[
 \mathcal A_j=
 \{z\in\mathcal X_j:
 z_n\text{ is }\F_n\text{-measurable}\}.
\]
The norm is
\[
 \norm{\left(\sum_n|z_n|^2\right)^{1/2}}_{p_j(\cdot)}.
\]
The adapted spaces are closed, since norm convergence implies
coordinatewise convergence in measure and admits almost everywhere
convergent subsequences in each coordinate.

\begin{corollary}\label{cor:adapted}
If $B=\infty$, then for every $0<\theta<1$ and $0<q\leq\infty$,
\[
 (\mathcal A_0,\mathcal A_1)_{\theta,q}
 \subsetneq
 (\mathcal X_0,\mathcal X_1)_{\theta,q}
 \cap\{\text{adapted sequences}\}.
\]
For fixed $\theta$, the same sequence witnesses strictness for
every $q$. Each $\mathcal A_j$ is nevertheless complemented in
$\mathcal X_j$.
\end{corollary}

\begin{proof}
Define
\[
 Jf=(f,0,0,\ldots),\qquad Qz=z_1.
\]
Both are endpoint contractions between the scalar and vector
couples, and $QJ$ is the identity. On the adapted subcouples they
map $Y_j$ to $\mathcal A_j$ and $\mathcal A_j$ to $Y_j$,
respectively. Applying these maps to decompositions in both
directions gives the exact identities
\[
 K(t,Jf;\mathcal X_0,\mathcal X_1)
 =
 K(t,f;X_0,X_1),
\]
\[
 K(t,Jf;\mathcal A_0,\mathcal A_1)
 =
 K(t,f;Y_0,Y_1).
\]
Now take $f=f_\theta$ from \eqref{eq:strongfailure}. The sequence
$Jf_\theta$ is adapted and has the required membership and
nonmembership properties.

Let $R_j:X_j\to Y_j$ be the positive projection from
Theorem~\ref{thm:main}. Then
\[
 \mathcal R_j(z_1,z_2,\ldots)
 =
 (R_jz_1,z_2,\ldots)
\]
is a projection onto $\mathcal A_j$. The pointwise estimate
\[
 \left(|R_jz_1|^2+\sum_{n\geq2}|z_n|^2\right)^{1/2}
 \leq
 |R_jz_1|+
 \left(\sum_{n\geq2}|z_n|^2\right)^{1/2}
\]
and the lattice property give
\[
 \norm{\mathcal R_jz}_{\mathcal X_j}
 \leq
 (\norm{R_j}+1)\norm z_{\mathcal X_j}.
\]
\end{proof}

\section{Further questions}\label{sec:scope}

For general exponent pairs, atomic oscillation control remains
sufficient by Lemma~\ref{lem:average}. It need not be necessary:
in the finite-profile model, a common maximizing coordinate gives
a bounded projection even when each exponent has a fixed positive
oscillation on atoms of arbitrarily small measure. It would be
useful to find a criterion expressed in terms of the joint
distribution of $p_0$ and $p_1$ on each atom.

The adapted example uses the filtration
$\F_1=\G$, $\F_n=\F$ for $n\geq2$, and a sequence with only
its first coordinate nonzero. The corresponding questions for
a fixed regular dyadic filtration and for martingale differences
remain open here. Conditional square-function Hardy spaces also
require a separate analysis because their norms involve
conditional expectations.

\section*{Declarations}

\paragraph{Funding.}
No funding was received for this work.

\paragraph{Competing interests.}
The author declares no competing interests.

\paragraph{Data availability.}
No datasets were generated or analyzed. All constructions and
proofs are contained in the article.

\paragraph{Declaration of generative AI use.}

The author used ChatGPT during the preparation
of this manuscript to support language editing, and LaTeX formatting. All content was subsequently
reviewed and revised by the author, who takes full responsibility for the published article.

\end{document}